\documentclass[12pt]{article}
\usepackage{amsfonts}
\usepackage{amsmath}
\usepackage{amssymb}
\usepackage{amsthm}
\usepackage{enumerate}
\makeatletter
\def\imod#1{\allowbreak\mkern10mu({\operator@font mod}\,\,#1)}
\makeatother
\DeclareMathOperator{\ord}{ord}

\newtheorem{theorem}{Theorem}[section]
\newtheorem{lemma}{Lemma}[section]

\newtheorem{conjecture}{Conjecture}[section]

\theoremstyle{definition}

\newtheorem{remark}{Remark}[section]

\begin{document}
\begin{center}
\vskip 1cm{\LARGE\bf On Schemmel Nontotient Numbers
\vskip 1cm
\large
Colin Defant\\
Department of Mathematics\\
University of Florida\\
United States\\
cdefant@ufl.edu}
\end{center}
\vskip .2 in

\begin{abstract} 
For each positive integer $r$, let $S_r$ denote the $r^{th}$ Schemmel totient function, a multiplicative arithmetic function defined by 
\[S_r(p^{\alpha})=\begin{cases} 0, & \mbox{if } p\leq r; \\ p^{\alpha-1}(p-r), & \mbox{if } p>r \end{cases}\] 
for all primes $p$ and positive integers $\alpha$. The function $S_1$ is simply Euler's totient function $\phi$. We define a Schemmel nontotient number of order $r$ to be a positive integer that is not in the range of the function $S_r$. In this paper, we modify several proofs due to Zhang in order to illustrate how many of the results currently known about nontotient numbers generalize to results concerning Schemmel nontotient numbers. We also invoke Zsigmondy's Theorem in order to generalize a result due to Mendelsohn. 
\end{abstract} 
\bigskip

\noindent \emph{Keywords: } Nontotient; Schemmel totient.

\noindent 2010 {\it Mathematics Subject Classification}:  Primary 11A25; Secondary 11N64.

\section{Introduction} 
Integers in the range of Euler's totient function $\phi$ are known as totient numbers, and positive integers that are not totient numbers are known as nontotient numbers. The study of nontotient numbers has burgeoned in the past sixty years due to contributors such as Schinzel, Ore, Selfridge, Mendelsohn, and Zhang. 
\par 
In 1869, V. Schemmel introduced a class of functions $S_r$, now known as Schemmel totient functions, that generalize Euler's totient function \cite{schemmel69}. For each positive integer $r$, $S_r$ is a multiplicative function that satisfies 
\[S_r(p^{\alpha})=\begin{cases} 0, & \mbox{if } p\leq r \\ p^{\alpha-1}(p-r), & \mbox{if } p>r \end{cases}\] 
for all primes $p$ and positive integers $\alpha$. We will make use of the fact that $S_r(x)\vert S_r(y)$ whenever $x$ and $y$ are positive integers such that $x\vert y$.  
\par 
For a positive integer $r$, we define a Schemmel totient number of order $r$ to be an integer in the range of the function $S_r$. Any positive integer that is not a Schemmel totient number of order $r$ is said to be a Schemmel nontotient number of order $r$. For convenience, we will let $G_r$ denote the set of Schemmel nontotient numbers of order $r$. Our goal is to generalize some of the results currently known about nontotient numbers to results concerning Schemmel nontotient numbers and to encourage further investigation of Schemmel nontotient numbers. In fairness to Ming Zhi Zhang, we note that many of the proofs presented here are merely adaptations of proofs given in \cite{Zhang93}. 
\par 
Many theorems deal with which nontotient numbers are divisible by certain powers of $2$, so we will explore two ways of generalizing such theorems. If $r$ is odd, then it is easy to see that all odd integers greater than $1$ are Schemmel nontotient numbers, Thus, when $r$ is odd, we will continue to pay attention to which Schemmel nontotient numbers are divisible by certain powers of $2$. On the other hand, if $r$ is even, then every even positive integer is a Schemmel nontotient number of order $r$. This follows from the fact that if $r>1$ and $S_r(n)>0$ for some positive integer $n$, then $n$ must be odd. Furthermore, it is easy to see that $S_r(n)$ is odd whenever $S_r(n)$ is positive, $n$ is odd, and $r$ is even. Thus, it is uninteresting to look at powers of $2$ dividing Schemmel nontotient numbers of order $r$ for even values of $r$. Instead, we will concentrate on values of $r$ for which $r+1$ is prime, and we will focus on the Schemmel nontotient numbers of order $r$ that are divisible by certain powers of $r+1$. 
\par 
For now, we prove one result for which the parity of $r$ is irrelevant. 
\begin{theorem} \label{Thm1.1} 
If $r$ and $m$ are positive integers, then there exist infinitely many primes $p$ such that $pm$ is a Schemmel nontotient number of order $r$.  
\end{theorem}     
\begin{proof} 
Fix positive integers $r$ and $m$, and let the positive divisors of $m$ be $d_1,d_2,\ldots,d_s$. Let $q_1,q_2,\ldots,q_s$ be primes satisfying $\max(m,r)<q_1<q_2<\cdots<q_s$. By the Chinese Remainder Theorem and Dirichlet's theorem concerning the infinitude of primes in arithmetic progressions, there are infinitely many primes $p>\max(q_s,m+r)$ that satisfy $d_ip\equiv -r\imod{q_i}$ for all $i\in\{1,2,\ldots,s\}$. Fix one such prime $p$, and suppose, for the sake of finding a contradiction, that $S_r(x)=pm$ for some positive integer $x$. If $p^2\vert x$, then $p(p-r)\vert S_r(x)=pm$, which contradicts the fact that $p>m+r$. If $p^2\nmid x$, then there must exist some prime $q$ such that $p\vert q-r$ and $q\vert x$. Then there exists some integer $d$ such that $pd=q-r\vert S_r(x)=pm$. This implies that $d=d_i$ for some $i\in\{1,2,\ldots,s\}$. Then, because $p$ satisfies the congruence $d_ip\equiv -r\imod{q_i}$, we see that $q_i\vert pd+r=q$, which implies that $q_i=q=pd+r$. However, this implies that $q_i>p$, which contradicts the fact that $p>q_s\geq q_i$.  
\end{proof} 
\par 
Throughout the remainder of this paper, we will let $\mathbb{N}$, $\mathbb{N}_0$, and $\mathbb{P}$ be the sets of positive integers, nonnegative integers, and prime numbers, respectively.  
\section{Schemmel Nontotient Numbers of Order One Less than a Prime} 
When $r+1$ is prime, it is particularly interesting to consider positive integers $k$ such that $(r+1)^\alpha k\in G_r$ for all nonnegative integers $\alpha$. To do so, we first establish the following two lemmata, the first of which generalizes a theorem due to Mendelsohn \cite{Mendelsohn76}.  
\begin{lemma} \label{Lem2.1} 
Let $m$ be a positive integer such that $m+1$ is not a power of $2$. If there exist positive integers $N,p_1,p_2$ such that $p_1$ and $p_2$ are distinct primes and $\ord_{p_1}(m)=\ord_{p_2}(m)=2^N$, then there exists an arithmetic progression $A$ with the following  three properties:
\begin{enumerate}[(a)]
\item $A$ contains infinitely many prime terms.
\item The common difference of $A$ is a product of $N+1$ distinct primes.
\item If $x$ is a term of $A$ and $t$ is a nonnegative integer, then $m^tx+m-1$ is divisible by exactly one of the $N+1$ prime divisors of the common difference of $A$.  
\end{enumerate}  
\end{lemma} 
\begin{proof} 
Zsigmondy's Theorem tells us that, for each positive integer $n$, there exists some prime that divides $m^{2^n}-1$ and does not divide $m^k-1$ for all positive integers $k<2^n$. In other words, for each positive integer $n$, we may find a prime $q_n$ such that $\ord_{q_n}(m)=2^n$. Suppose there exist positive integers $N,p_1,p_2$ such that $p_1$ and $p_2$ are distinct primes and $\ord_{p_1}(m)=\ord_{p_2}(m)=2^N$. Without loss of generality, we may let $q_N=p_1$ and write $q_0=p_2$. Let $\displaystyle{M=\prod_{i=0}^N}q_i$, and consider the system of congruences 
\begin{equation} \label{Eq2.1} 
\begin{cases} x+m\equiv 1\imod{q_n}, & \mbox{if } n=0 \\ m^{2^{n-1}}x+m\equiv 1\imod{q_n}, & \mbox{if } n\in\{1,2,\ldots,N\}. \end{cases}
\end{equation} 
The Chinese Remainder Theorem tells us that the positive solutions to \eqref{Eq2.1} are precisely the terms of an arithmetic progression $A=a,a+M,a+2M,\ldots$ for some positive integer $a<M$. In addition, any solution to \eqref{Eq2.1} is relatively prime to $M$ because $q_n\nmid m-1$ for all $n\in\{0,1,\ldots,N\}$. Therefore, Dirichlet's theorem concerning the infinitude of primes in arithmetic progressions guarantees that $A$ has infinitely many prime terms. 
\par 
Now, choose some term $x$ of $A$, and let $t$ be a nonnegative integer. We will show that $q_n\vert m^tx+m-1$ for precisely one $n\in\{0,1,\ldots,N\}$. First, let $n\in\{1,2,\ldots,N\}$. We may use the fact that $m^{2^{n-1}}x+m\equiv 1\imod{q_n}$ to conclude that $q_n\vert m^tx+m-1$ if and only if $m^t\equiv m^{2^{n-1}}\imod{q_n}$. Furthermore, because $\ord_{q_n}(m)=2^n$, we see that $m^t\equiv m^{2^{n-1}}\imod{q_n}$ if and only if $t\equiv 2^{n-1}\imod{2^n}$. Similarly, because $x+m\equiv 1\imod{q_0}$, we see that $q_0\vert m^tx+m-1$ if and only if $m^t\equiv 1\imod{q_0}$. Because $\ord_{q_0}(m)=2^N$, we find that $m^t\equiv 1\imod{q_0}$ if and only if $t\equiv 0\imod{2^N}$. If $t=0$, it is clear that $q_0$ is the only element of the set $Q=\{q_0,q_1,\ldots,q_N\}$ that divides $m^tx+m-1$, so we may assume $t>0$. If we write $t=2^{\beta}\mu$, where $\beta,\mu\in\mathbb{N}_0$ and $2\nmid\mu$, then $\beta$ completely determines which primes in $Q$ divide $m^tx+m-1$. If $\beta\geq N$, then $q_0$ is the only element of $Q$ that divides $m^tx+m-1$. If $\beta<N$, then $q_{\beta+1}$ is the only element of $Q$ that divides $m^tx+m-1$. This completes the proof. 
\end{proof}  
\begin{lemma} \label{Lem2.2} 
If $r$, $\alpha$, and $p$ are nonnegative integers with $r+1,p\in\mathbb{P}$, then $(r+1)^{\alpha}p\in G_r$ if and only if $p\neq (r+1)^t+r$ and $(r+1)^tp+r\not\in\mathbb{P}$ for all nonnegative integers $t\leq\alpha$.   
\end{lemma} 
\begin{proof} 
First, suppose $S_r(x)=(r+1)^{\alpha}p$ for some positive integer $x$. If $p^2\vert x$, then $p(p-r)\vert S_r(x)=(r+1)^{\alpha}p$, so $p-r=(r+1)^t$ for some nonnegative integer $t\leq\alpha$. On the other hand, if $p^2\nmid x$, then there exists some prime $q$ such that $p\vert q-r$ and $q\vert x$. This implies that $q-r\vert S_r(x)=(r+1)^{\alpha}p$, which means that $q-r=(r+1)^tp$ for some nonnegative integer $t\leq\alpha$. Thus, if $p\neq (r+1)^t+r$ and $(r+1)^tp+r\not\in\mathbb{P}$ for all nonnegative integers $t\leq\alpha$, then $(r+1)^{\alpha}p\in G_r$. 
\par 
To prove the converse, suppose $p=(r+1)^t+r$ or $(r+1)^tp+r\in\mathbb{P}$ for some nonnegative integer $t\leq\alpha$. If $p=(r+1)^t+r$ and $t>0$, then $S_r((r+1)^{\alpha-t+1}p^2)=S_r((r+1)^{\alpha-t+1})S_r(p^2)=(r+1)^{\alpha-t}p(p-r)=(r+1)^{\alpha}p$. If $p=(r+1)^t+r$ and $t=0$, then $S_r((r+1)^{\alpha+2})=(r+1)^{\alpha+1}=(r+1)^{\alpha}p$. If $(r+1)^tp+r\in\mathbb{P}$, then $S_r((r+1)^{\alpha-t+1}((r+1)^tp+r))=S_r((r+1)^{\alpha-t+1})S_r((r+1)^tp+r)=(r+1)^{\alpha}p$. Thus, if $(r+1)^{\alpha}p\in G_r$, then we must have $p\neq (r+1)^t+r$ and $(r+1)^tp+r\not\in\mathbb{P}$ for all nonnegative integers $t\leq\alpha$. 
\end{proof} 
\begin{theorem} \label{Thm2.1} 
Suppose $r+1$ is a prime that is not a Mersenne prime. If there exist integers $N,p_1,p_2$ such that $p_1$ and $p_2$ are distinct primes and $\ord_{p_1}(r+1)=\ord_{p_2}(r+1)=2^N$, then there are infinitely many primes $p$ such that $(r+1)^{\alpha}p\in G_r$ for all nonnegative integers $\alpha$.  
\end{theorem} 
\begin{proof} 
Suppose that there exist integers $N,p_1,p_2$ such that $p_1$ and $p_2$ are distinct primes and $\ord_{p_1}(r+1)=\ord_{p_2}(r+1)=2^N$. We will show that there are infinitely many primes $p$ such that $p\neq (r+1)^t+r$ and $(r+1)^tp+r\not\in\mathbb{P}$ for all nonnegative integers $t$, from which Lemma \ref{Lem2.2} will yield the desired result. We may use Lemma \ref{Lem2.1} to conclude that there exists an arithmetic progression $A$ that has infinitely many prime terms and has common difference $\displaystyle{M=\prod_{i=0}^N}q_i$, where $q_0,q_1,\ldots,q_N$ are distinct primes. Furthermore, Lemma \ref{Lem2.1} tells us that if $p>M$ is a prime term of $A$ and $t$ is any nonnegative integer, then $(r+1)^tp+r$ is composite because it is divisible by one of the $N+1$ prime divisors of $M$. Hence, it suffices to show that there are infinitely many prime terms $p$ of $A$ that are not of the form $(r+1)^t+r$. 
\par 
If we let $\pi(x;M,a)$ denote the number of prime terms of $A$ that are less than or equal to $x$, then the Prime Number Theorem extended to arithmetic progressions tells us that $\displaystyle{\pi(x;M,a)\sim\frac{1}{\phi(M)}\frac{x}{\log x}}$ as $x\rightarrow\infty$. Because the number of primes less than or equal to $x$ of the form $(r+1)^t+r$ is clearly of order $\displaystyle{o\left(\frac{x}{\log x}\right)}$, the proof is complete.    
\end{proof} 
\begin{theorem} \label{Thm2.2}
Suppose that $r+1$ is a prime that is not a Mersenne prime and that there exist integers $N,p_1,p_2$ such that $p_1$ and $p_2$ are distinct primes and $\ord_{p_1}(r+1)=\ord_{p_2}(r+1)=2^N$. Let $M$ be as in the proof of Theorem \ref{Thm2.1}. Suppose $B=p_1^{\alpha_1}p_2^{\alpha_2}\cdots p_s^{\alpha_s}$, where $p_1,p_2,\ldots,p_s$ are distinct primes that are each greater than $r$ and congruent to $1$ modulo $M$ and $\alpha_1,\alpha_2,\ldots,\alpha_s$ are positive integers. If $p>M$ is one of the infinitely many primes that satisfies $(r+1)^{\alpha}p\in G_r$ for all $\alpha\in\mathbb{N}_0$ and $p-r$ has a prime divisor $P$ that does not divide $(r+1)B$, then $(r+1)^{\alpha}Bp\in G_r$ for all nonegative integers $\alpha$.  
\end{theorem} 
\begin{proof} 
Suppose $S_r(x)=(r+1)^{\alpha}Bp$ for some nonnegative integers $\alpha$ and $x$. The existence of $P$ guarantees that $p^2\nmid x$, so there is some prime $q$ such that $p\vert q-r$ and $q\vert x$. Then $q=pd+r$ for some positive integer $d$, so $pd=q-r\vert (r+1)^{\alpha}Bp$. This implies that there exist nonnegative integers $t,\gamma_1,\gamma_2,\ldots,\gamma_s$ such that $pd+r=(r+1)^tpp_1^{\gamma_1}p_2^{\gamma_2}\cdots p_s^{\gamma_s}+r\equiv(r+1)^tp+r\imod{M}$. By Lemma \ref{Lem2.1}, there exists a unique prime divisor $q_i$ of $M$ that divides $(r+1)^tp+r$, so $q_i\vert pd+r=q$. This implies that $q=q_i$, which contradicts the fact that $q=pd+r>Md+r>q_i$.  
\end{proof} 
\begin{theorem} \label{Thm2.3} 
Suppose $r+1$ is a prime and $k$ is a positive integer such that $r+1\nmid k$ and $(r+1)^{\alpha}k\in G_r$ for all nonnegative integers $\alpha$. If $k_1$ and $k_2$ are relatively prime positive integers such that $k_1k_2=k$, then either $(r+1)^{\alpha}k_1\in G_r$ for all nonnegative integers $\alpha$ or $(r+1)^{\alpha}k_2\in G_r$ for all nonnegative integers $\alpha$. 
\end{theorem} 
\begin{proof} 
Suppose, for the sake of finding a contradiction, that there exist nonnegative integers $k_1,k_2,\alpha_1,\alpha_2,x_1,x_2$ such that $\gcd(k_1,k_2)=1$, $k_1k_2=k$, $S_r(x_1)=(r+1)^{\alpha_1}k_1$, and $S_r(x_2)=(r+1)^{\alpha_2}k_2$. We may assume that $\alpha_1$ and $\alpha_2$ are minimal with respect to these properties. Suppose $p=(r+1)^t+r$ is prime for some positive integer $t$. If $p^1\parallel x_1$, then we may write $x_1=p\mu$, where $\mu\in\mathbb{N}$ and $p\nmid\mu$. We then have $S_r(x_1)=(p-r)S_r(\mu)=(r+1)^tS_r(\mu)=(r+1)^{\alpha_1}k_1$, so $t\leq\alpha_1$ and $S_r(\mu)=(r+1)^{\alpha_1-t}k_1$, which contradicts the minimality of $\alpha_1$. Thus, if $p\vert x_1$, then $p^2\vert x_1$. By the same token, if $p\vert x_2$, then $p^2\vert x_2$. Let us write $d=\gcd(x_1,x_2)$ so that $S_r(d)\vert\gcd(S_r(x_1),S_r(x_2))=(r+1)^{\min(\alpha_1,\alpha_2)}$. Then $S_r(d)=(r+1)^{\beta}$ for some nonnegative integer $\beta$, so we may write $d=(r+1)^{\gamma}\lambda$, where $\gamma\in\mathbb{N}_0$ and $\lambda$ is a (possibly empty) product of distinct primes of the form $(r+1)^t+r$ ($t\in\mathbb{N}$). If $p=(r+1)^t+r$ is prime for some $t\in\mathbb{N}$ and $p\vert\lambda$, then $p\vert x_1,x_2$. This implies that $p^2\vert x_1,x_2$, so $p^2\vert d$. However, this contradicts the fact that $\lambda$ is a product of distinct primes, so we conclude that $\lambda=1$. Either $(r+1)^{\gamma+1}\nmid x_1$ or $(r+1)^{\gamma+1}\nmid x_2$, so we may assume, without loss of generality, that $(r+1)^{\gamma+1}\nmid x_1$. Then $x_1=(r+1)^{\gamma}y$, where $y\in\mathbb{N}$ and $r+1\nmid y$. Because $(r+1)^{\alpha_1}k_1=S_r(x_1)=S_r((r+1)^{\gamma})S_r(y)=(r+1)^{\gamma-1}S_r(y)$, we find that $S_r(y)=(r+1)^{\alpha_1-\gamma+1}k_1$. However, $x_2$ and $y$ are relatively prime, so $S_r(x_2y)=S_r(x_2)S_r(y)=(r+1)^{\alpha_1+\alpha_2-\gamma+1}k_1k_2=(r+1)^{\alpha_1+\alpha_2-\gamma+1}k$. This contradicts the hypothesis that $(r+1)^{\alpha}k\in G_r$ for all nonnegative integers $\alpha$, so the proof is complete.    
\end{proof} 
\section{Schemmel Nontotient Numbers \\ of Odd Order}
\begin{theorem} \label{Thm3.1} 
Let $r$ be an odd positive integer, and write $n=2p_1^{\alpha_1}p_2^{\alpha_2}\cdots p_s^{\alpha_s}$, where, for all $i,j\in\{1,2,\ldots,s\}$ with $i<j$, $p_i$ and $p_j$ are odd primes, $\alpha_i$ and $\alpha_j$ are positive integers, and $p_i<p_j$. Then $n$ is a Schemmel nontotient number of order $r$ if and only if $n+r$ is composite and $p_s-r\neq 2p_1^{\alpha_1}p_2^{\alpha_2}\cdots p_{s-1}^{\alpha_{s-1}}$. 
\end{theorem} 
\begin{proof} 
Zhang has proven the case $r=1$ \cite{Zhang93}, so we may assume that $r\geq 3$. First, suppose $S_r(x)=n$ for some positive integer $x$. Note that $S_r(p^{\alpha})$ is even for any prime $p$ and positive integer $\alpha$. Therefore, we know that $\omega(x)=1$, so we may write $x=p^{\alpha}$ for some prime $p$ and positive integer $\alpha$. If $\alpha=1$, then $n+r=p$. If $\alpha>1$, then we must have $p=p_s$ and $\alpha-1=\alpha_s$ because $n=p^{\alpha-1}(p-r)$. This implies that we must have $p_s-r=p-r=2p_1^{\alpha_1}p_2^{\alpha_2}\cdots p_{s-1}^{\alpha_{s-1}}$. Hence, if $n+r$ is composite and $p_s-r \neq 2p_1^{\alpha_1}p_2^{\alpha_2}\cdots p_{s-1}^{\alpha_{s-1}}$, then $n\in G_r$. 
\par 
Conversely, suppose that $n+r$ is prime or $p_s-r=2p_1^{\alpha_1}p_2^{\alpha_2}\cdots p_{s-1}^{\alpha_{s-1}}$. If $n+r$ is prime, then $S_r(n+r)=n$, so $n\not\in G_r$. If $p_s-r=2p_1^{\alpha_1}p_2^{\alpha_2}\cdots p_{s-1}^{\alpha_{s-1}}$, then $S_r(p_s^{\alpha_s+1})=n$, so $n\not\in G_r$. 
\end{proof} 
\begin{theorem} \label{Thm3.2}
Suppose $r$ is an odd positive integer and $n=2^{\alpha}p_1^{\alpha_1}p_2^{\alpha_2}\cdots p_s^{\alpha_s}$, where $p_1,p_2,\ldots,p_s$ are distinct odd primes that are each greater than $r$, $\alpha,\alpha_1,\alpha_2,\ldots,\alpha_s$ are positive integers, and $2^tp_1^{\gamma}+r$ is composite for all $t\in\{1,2,\ldots,\alpha\}$ and $\gamma\in\{1,2,\ldots,\alpha_1\}$. For each $t\in\{1,2,\ldots,\alpha\}$ and $\gamma\in\{1,2,\ldots,\alpha_1\}$, let $q_{t,\gamma}$ be a prime divisor of $2^tp^{\gamma}+r$, and let $M$ be the least common multiple of all such $q_{t,\gamma}$. If $p_i\equiv 1\imod{M}$ for all $i\in\{2,3,\ldots,s\}$ and $p_1-r\nmid n$, then $n$ is a Schemmel nontotient number of order $r$. 
\end{theorem} 
\begin{proof} 
Suppose $S_r(x)=n$ for some positive integer $x$. Because $p_1-r\nmid n$, we see that $p_1^2\nmid x$. Thus, there exists a prime $q$ such that $p_1\vert q-r$ and $q\vert x$. Then $q=p_1d+r$ for some positive integer $d$, and we have $p_1d\vert 2^{\alpha}p_1^{\alpha_1}p_2^{\alpha_2}\cdots p_s^{\alpha_s}$. We may write $p_1d+r=2^tp_1^{\gamma}p_2^{\gamma_2}\cdots p_s^{\gamma_s}+r\equiv 2^tp_1^{\gamma}+r\imod{M}$, so $q=p_1d+r\equiv 2^tp_1^{\gamma}+r\equiv 0\imod{q_{t,\gamma}}$. This implies that $q=q_{t,\gamma}$, which contradicts the fact that $q=p_1d+r\geq2^tp_1^{\gamma}+r>q_{t,\gamma}$. 
\end{proof} 
\begin{remark} \label{Rem3.1} 
Even within the case $r=1$, Theorem \ref{Thm3.2} provides a slight generalization of Theorem $3$ in \cite{Zhang93}.
\end{remark} 
The proof of the following theorem utilizes Mendelsohn's original observations concerning the orders of the number $2$ modulo certain primes \cite{Mendelsohn76}. 
\begin{theorem} \label{Thm3.3} 
Suppose $r$ is an odd positive integer not divisible by $3$, $5$, $17$, $257$, $641$, $65537$, or $6700417$. There exist infinitely many primes $p$ such that $2^{\alpha}p$ is a Schemmel nontotient number of order $r$ for all nonnegative integers $\alpha$. 
\end{theorem} 
\begin{proof} 
Let use write $q_0=6700417$, $q_1=3$, $q_2=5$, $q_3=17$, $q_4=257$, $q_5=65537$, and $q_6=641$. Observe that $\ord_{q_n}(2)=2^n$ for each positive integer $n\leq 6$ and $\ord_{q_0}(2)=2^6$. Now consider the system of congruences 
\begin{equation} \label{Eq3.1} 
\begin{cases} x+r\equiv 0\imod{q_n}, & \mbox{if } n=0 \\ 2^{2^{n-1}}x+r\equiv 0\imod{q_n}, & \mbox{if } n\in\{1,2,\ldots,6\}. \end{cases}
\end{equation} 
By the Chinese Remainder Theorem, the positive solutions to \eqref{Eq3.1} are precisely the terms of an arithmetic progression with common difference \\ 
$\displaystyle{M=\prod_{i=0}^6q_i}$. Furthermore, any solution to \eqref{Eq3.1} is relatively prime to $M$ because $q_n\nmid r$ for all $n\in\{0,1,\ldots,6\}$. Following an argument virtually identical to that used in the proof of Lemma \ref{Lem2.1}, we see that if $x$ is any solution to the system of congruences \eqref{Eq3.1}, then, for any nonnegative integer $t$, $2^tx+r$ is divisible by precisely one element of the set $\{q_0,q_1,\ldots,q_6\}$. Furthermore, there are infinitely many prime solutions to \eqref{Eq3.1} that are not of the form $2^t+r$ ($t\in\mathbb{N}_0$). Therefore, there are infinitely many primes $p$ such that $p\neq 2^t+r$ and $2^tp+r\not\in\mathbb{P}$ for all nonnegative integers $t$. Fix one such prime $p$ and suppose, for the sake of finding a contradiction, that $S_r(x)=2^{\alpha}p$ for some nonnegative integers $x$ and $\alpha$. If $p^2\vert x$, then $p-r\vert2^{\alpha}$, which is a contradiction. If $p^2\nmid x$, then there exists some prime $q$ such that $q\vert x$ and $q-r=2^tp$ for some $t\in\mathbb{N}_0$, which is also a contradiction.   
\end{proof} 
\begin{theorem} \label{Thm3.4}
Let $r$ and $k$ be odd positive integers such that $2^{\alpha}k\in G_r$ for all nonnegative integers $\alpha$. Let $k_1$ and $k_2$ be relatively prime positive integers such that $k_1k_2=k$. Either $2^{\alpha}k_1\in G_r$ for all nonnegative integers $\alpha$ or $2^{\alpha}k_2\in G_r$ for all nonnegative integers $\alpha$
\end{theorem} 
\begin{proof} 
Zhang proves the case $r=1$ as Theorem 4 in \cite{Zhang93}, so we will assume $r>1$. Suppose that there exist nonnegative integers $\alpha_1,\alpha_2,x_1,x_2$ such that $S_r(x_1)=2^{\alpha_1}k_1$ and $S_r(x_2)=2^{\alpha_2}k_2$, and assume that $\alpha_1$ and $\alpha_2$ are minimal with respect to these properties. Note that $x_1$ and $x_2$ must be odd because $r>1$. Write $d=\gcd(x_1,x_2)$. Then $S_r(d)\vert\gcd(S_r(x_1),S_r(x_2))=2^{\min(\alpha_1,\alpha_2)}$, which means that $d$ is a (possibly empty) product of distinct primes that are each $r$ more than a power of $2$.  Suppose $t$ is a nonnegative integer such that $p=2^t+r$ is prime. If $p\vert x_1$, then we may write $x_1=p\mu$, where $\mu\in\mathbb{N}$ and $p\nmid\mu$. This implies that $2^{\alpha_1}k_1=S_r(x_1)=S_r(p)S_r(\mu)=2^tS_r(\mu)$, so $S_r(\mu)=2^{\alpha_1-t}k_1$, which contradicts the minimality of $\alpha_1$. We reach a similar contradiction if we assume $p\vert x_2$. Thus, $d=1$, so $S_r(x_1x_2)=S_r(x_1)S_r(x_2)=2^{\alpha_1+\alpha_2}k_1k_2=2^{\alpha_1+\alpha_2}k$. This contradicts the fact that $2^{\alpha}k\in G_r$ for all nonnegative integers $\alpha$.   
\end{proof} 
\section{Concluding Remarks} 
Clearly, we have only scratched the surface of the topic of Schemmel nontotient numbers, so we encourage the reader to continue the exploration. Indeed, except for Theorem \ref{Thm1.1}, we have not even considered Schemmel nontotient numbers of order $r$ when $r+1$ is odd and composite. Furthermore, after acknowledging the restrictions listed in the hypothesis of Theorem \ref{Thm2.1}, we make the following conjecture. 
\begin{conjecture} \label{Conj4.1} 
If $r+1$ is prime, then there are infinitely many primes $p$ such that $(r+1)^{\alpha}p\in G_r$ for all nonnegative integers $\alpha$. 
\end{conjecture} 
 
\end{document}